\documentclass[12pt, a4paper, reqno]{amsart}
\usepackage{color}
\usepackage[centering, heightrounded]{geometry}

\usepackage{amsmath, amsthm}
\usepackage{amssymb}
\usepackage{mathrsfs}%
\usepackage{braket}%

\usepackage{secdot}%

\usepackage{enumitem}
\setlist[enumerate]{topsep=2pt, parsep=0pt, partopsep=0pt, itemsep=3pt, leftmargin=3em, label=\rm{(\roman*)}}

\usepackage{cancel}

\let\oldleq\leq
\let\oldsubset\subset
\let\oldcap\cap
\let\oldin\in
\usepackage{mathabx}
\let\leq\oldleq
\let\subset\oldsubset
\let\cap\oldcap
\let\in\oldin

\numberwithin{equation}{section}

\newtheorem{theorem}{Theorem}[section]
\newtheorem{lemma}[theorem]{Lemma}
\newtheorem{proposition}[theorem]{Proposition}
\newtheorem{corollary}[theorem]{Corollary}
\newtheorem*{theorem*}{Theorem}
\newtheorem*{remark*}{Remark}
\newtheorem*{assumption*}{Assumption}

\theoremstyle{definition}
\newtheorem{definition}[theorem]{Definition}
\theoremstyle{remark}

\renewcommand{\l}{\left}
\renewcommand{\r}{\right}
\newcommand{\eps}{\varepsilon}
\newcommand{\N}{{\mathbb N}}
\newcommand{\R}{{\mathbb R}}
\newcommand{\C}{{\mathbb C}}

\newcommand{\Z}{{\mathbb Z}}
\newcommand{\T}{{\mathbb T}}

\renewcommand{\Re}{\operatorname{Re}}
\renewcommand{\Im}{\operatorname{Im}}

\newcommand{\ds}{\displaystyle}

\newcommand{\pt}{\partial}
\newcommand{\cleq}{\lesssim}

\newcommand{\wto}{\rightharpoonup}%

\def\tbra[#1,#2]{\left\langle #1 , #2\right\rangle} 
\def\rbra[#1,#2]{\left( #1 , #2 \right)} 

\newcommand{\ce}{\mathrel{\mathop:}=}

\def\norm[#1]{\left\Vert #1 \right\Vert}
\def\abs[#1]{\left\vert #1 \right\vert}

\newcommand{\cH}{{\mathcal H}}

\newcommand{\scG}{{\mathscr G}}

\begin{document}
\title[]{Global $H^2$-solutions for the generalized derivative NLS on $\T$}

\author[M. Hayashi]{Masayuki Hayashi}
\address{Graduate School of Human and Environmental Studies,
Kyoto University, Kyoto 606-8501, Japan.
\newline\indent
Waseda Research Institute for Science and Engineering, Waseda University, Tokyo 169-8555, Japan
}
 \email{hayashi.masayuki.3m@kyoto-u.ac.jp}

\author[T. Ozawa]{Tohru Ozawa}
\address{Department of Applied Physics\\ Waseda University\\ Tokyo
  169-8555\\  Japan}
\email{txozawa@waseda.jp}

\author[N. Visciglia]{Nicola Visciglia}
\address{Dipartimento di Matematica, Universit\`a di Pisa, Largo Bruno Pontecorvo, 5 56127 Pisa, Italy
}%
\email{nicola.visciglia@unipi.it}

\date{\today}

\begin{abstract}
We prove global existence of $H^2$ solutions to the Cauchy problem for the generalized derivative nonlinear Schr\"odinger equation on the 1-d torus. 
This answers an open problem posed by Ambrose and Simpson \cite{AS15}. 
The key is the extraction of the terms that cause the problem in energy estimates and the construction of suitable energies so as to cancel the problematic terms out by effectively using integration by parts and the equation.
\end{abstract}

\maketitle
\setcounter{tocdepth}{1}
\tableofcontents

\section{Introduction}
\label{sec:1}
We consider the Cauchy problem for the generalized derivative nonlinear Schr\"odinger equation
\begin{align}
\label{eq:1.1}
\l\{
\begin{aligned}
&i\pt_t u+\pt_x^2 u+i|u|^{2\sigma}\pt_xu=0,
\\
&u_{\mid t=0}=\varphi ,
\end{aligned}
\r.
\quad (t,x)\in\R\times\T,~\sigma>1,
\end{align}
where $\T\ce\R/\Z$.
The following quantities are formally invariant by the flow of \eqref{eq:1.1}:
\begin{align*}
\norm[u(t)]_{L^2(\T)}^2=\norm[\varphi]_{L^2(\T)}^2,
\quad E(u(t))=E(\varphi),
\end{align*}
where the energy $E(u)$ is defined by
\begin{align*}
E(u)=\frac{1}{2}\norm[\pt_xu]_{L^2(\T)}^2+\frac{1}{2\sigma+2}
\Im\int_{\T}|u|^{2\sigma}\pt_xu\widebar{u}dx.
\end{align*}
When $\sigma=1$, the equation corresponds to the standard derivative nonlinear Schr\"odinger equation, which is known to be  completely integrable (\cite{KN78}). There is a vast literature in this case and here we only refer to the recent results \cite{KVpre, H21} and references therein.
In this paper we are interested in the case $\sigma>1$ including noninteger powers. We note the scaling property: if we consider the equation \eqref{eq:1.1} on the line $\R$, the equation is invariant under the transformation
\begin{align*}
u_{\lambda}(t,x)= \lambda^{\frac{1}{2\sigma}} u(\lambda^2 t,\lambda x),\quad\lambda>0,
\end{align*}
which implies that the critical Sobolev exponent is $s_c=\frac{1}{2}-\frac{1}{2\sigma}$. In particular, when $\sigma>1$, the equation is $L^2$ supercritical.

The equation \eqref{eq:1.1} has attracted attention since the interesting numerical results \cite{LSS13a, LSS13b} by Liu, Simpson and Sulem.
The mathematical study of \eqref{eq:1.1} has been considered, regarding the Cauchy problem \cite{Hao07, AS15, S15, HO16, LPS19a, LPS19b, PTpre}, global properties of solutions \cite{CSS17, FHI17, BWX20}, and stability/instability of solitary waves \cite{LSS13a, F17, G18, LNpre}. We note that most of these results are on the line.  
Ambrose and Simpson \cite{AS15} proved that for any $\varphi\in H^2(\T)$ there exists a unique solution $u\in C([0,T_{\max}),H^2(\T))$ of \eqref{eq:1.1}, where $[0,T_{\rm max})$ is the maximal existence interval of the $H^2(\T)$ solution, and that the standard blowup alternative holds: $T_{\rm max}=\infty$, or $T_{\rm max}<\infty$ implies $\lim_{t\to T_{\rm max}}\norm[u(t)]_{H^2(\T)}=\infty$. The main results of \cite{AS15} concern the local Cauchy theory by a compactness argument,
but the global existence of $H^2(\T)$ solutions has remained unsolved. In this paper we study the global Cauchy problem  for \eqref{eq:1.1} in the $H^2(\T)$ setting.

Our main result is the following.
\begin{theorem}
\label{thm:1.1}
Let $\varphi\in H^2(\T)$. For the maximal solution $u\in C([0,T_{\rm max}), H^2(\T))$ to \eqref{eq:1.1}, we have the following alternative:
\begin{enumerate}
\item $T_{\rm max}=\infty$,

\item $T_{\rm max}<\infty$ implies $\limsup_{t\uparrow T_{\rm max}}\norm[u(t)]_{H^1(\T)}=\infty$. 
\end{enumerate}
The same alternative also holds true for the negative time direction.
\end{theorem}

As a corollary, we prove the following global existence of $H^2(\T)$ solutions for \eqref{eq:1.1} under the smallness condition on the initial data, which proves the conjecture in \cite[Section 5]{AS15}.
\begin{corollary}
\label{cor:1.2}
There exists $\delta>0$ such that if $\varphi\in H^2(\T)$ satisfies $\norm[\varphi]_{H^1(\T)}<\delta$, then there exists an unique global solution $u\in C(\R, H^2(\T))$ to \eqref{eq:1.1}.
\end{corollary}
When $\sigma=1$, our results may be considered to correspond to \cite[Theorem 2]{TF81}, whose proof, however, relies on the $H^2$ conservation law that follows from the integrability structure. We cannot expect an integrability structure when $\sigma>1$, so the problem becomes much more delicate.

The main difficulty in order to 
get a global existence result in $H^2(\T)$ is to establish an energy estimate that allows to apply a globalization argument together with Gronwall's lemma. Indeed it is not difficult to check that for solutions to \eqref{eq:1.1}
the following estimate holds:
\begin{align*}
\frac d{dt} \|u\|_{H^2(\T)}^2 \leq C \|u\|_{H^2(\T)}^2\|u\|_{W^{1,\infty}(\T)} \|u\|_{H^1(\T)}^{2\sigma-1}. 
\end{align*}
However, this estimate is useless for the desired globalization even if we assume an a priori uniform bound on the $H^1(\T)$ norm of the solution. Our idea is to compute a more sophisticated energy ${\mathcal E}(u)$
(see Theorem \ref{thm:2.4} and Section \ref{sec:4}), which at leading order is equivalent to the $H^2(\T)$-norm such that 
we get the bound:
\begin{align*}
\abs[ \frac d{dt} {\mathcal E}(u)]  \leq C \|u\|_{H^2(\T)}^2 f(\|u\|_{H^{1}(\T)}),
\end{align*}
where $f\in {\mathcal C}(\R, \R)$.
The idea to modify high Sobolev norms with lower order perturbations, in order to get 
cancellation of the bad interaction along the computation of the associated energy estimate, has been extensively used in the literature, we quote for shortness  a few of them \cite{CO15, G89, HO87, OV16, PTV17, T17, V21}. From a technical viewpoint in order to justify the manipulations that we need to do, we have to work on the regularized equation associated with \eqref{eq:1.1} which admits smooth solutions and hence we can compute at that level all the derivatives that we need, and at the end we transfer those bounds at the level of the original equation \eqref{eq:1.1}. 
Since the nonlinear terms involve derivatives and non-integer powers, the construction and justification of the modified energy ${\mathcal E}(u)$ requires a more delicate discussion than previous literature.

The argument in this paper holds for the case of the line in the same way. However, in the case of the line, Theorem \ref{thm:1.1} can be easily proven by applying the wellposedness result of $H^1(\R)$. Indeed, it is proved in \cite{HO16} that for any initial data $\varphi\in H^1(\R)$ there exists a unique solution in the class
\begin{align}
\label{eq:1.2}
C([0,T], H^1(\R)) \cap L^4([0,T], W^{1,\infty}(\R))\quad\text{for}~T=T(\norm[\varphi]_{H^1(\R)}) 
\end{align}
for the equation \eqref{eq:1.1} on the line. This enables us to control the time integral of the norm $\norm[u]_{W^{1,\infty}(\R)}$ from the boundedness of $H^1(\R)$ norm of the solution, and together with the energy estimates in $H^2(\R)$, one can prove the theorem. The construction of solutions in the class of \eqref{eq:1.2} is obtained by combining gauge transformations and Strichartz estimates, which is inspired by the works \cite{H93, HO92, HO94a} for the standard derivative NLS equation ($\sigma=1$).

In the case of the torus, it is known that Strichartz estimates involve a loss of derivatives (see \cite{B93, BGT04}), so we cannot expect a solution to be constructed in the class of \eqref{eq:1.2} rewritten to the torus. We note that the $H^1(\T)$ wellposedness for general $\sigma>1$ remains an open problem.

The rest of the paper is organized as follows. In Section \ref{sec:2} we introduce the approximate equation for \eqref{eq:1.1} and compute suitable energies in the $H^2(\T)$ setting. The main purpose of this section is to derive the $H^2(\T)$ identity for the approximate equation, which is the key to the proof of our main theorem. In Section \ref{sec:3} we prove the global existence of $H^2(\T)$ solutions to \eqref{eq:1.1}, based on the modified energy identity in the previous section. In the case of large initial data, in order to obtain the uniform $H^1$ boundedness of approximate solutions, we use a somewhat delicate argument, such as dividing the time interval and extending the solution in a finite number of times (see Section \ref{sec:3.5} for the necessity of this argument).
We see that the uniform estimate in $H^s$ for $s\in(3/2,2)$ are useful in this argument, and provide a self-contained proof of this estimate in Appendix \ref{sec:A}, which may be of independent interest. In Section \ref{sec:4} we explain how modified energies in the key $H^2(\T)$ identity are derived from a heuristic discussion.


\subsection*{Notation} 
For $f,g\in L^2(\T)$, the standard inner product is defined by
\begin{align*}
\rbra[f,g]_{L^2(\T)} =\int_\T f(x)\widebar{g(x)} dx.
\end{align*}
The Fourier transform on the torus is defined by
\begin{align*}
\hat{f}(n) =\int_{\T} f(x) e^{-2\pi inx}dx,\quad n\in\Z.
\end{align*}
The Sobolev spaces $H^s(\T)$ on the torus are defined via the norm
\begin{align*}
\norm[f]_{H^s(\T)}^2 = \sum_{n\in\Z}\l(1+|n|^2\r)^s |\hat f(n)|^2\quad\text{for}~s\in\R
\end{align*}
and $H^\infty(\T)\ce \bigcap_{m\in\N} H^m(\T)$.
The homogeneous Sobolev spaces are defined in a similar way:
\begin{align*}
\norm[f]_{\dot{H}^s(\T)}^2 = \sum_{n\in\Z}|n|^{2s} |\hat f(n)|^2\quad\text{for}~s\in\R.
\end{align*}
From the next section onwards, we will write
\begin{align*}
H^s=H^s(\T),\quad L^p=L^p(\T)
\end{align*}
for every $s\in\R$ and every $p\in[1,\infty]$.
We may also write $\pt=\pt_x$ and  
\begin{align*}
\int v =\int_{\T} v(t)=\int_{\T} v(t,x)dx
\end{align*}
for any time-dependent function $v(t,x)$.

We use $A\cleq B$ to denote the inequality $A\leq CB$ for some constant $C>0$. The dependence of $C$ is usually clear from the context and we often omit this dependence. We sometimes denote by $C=C(*)$ a constant depending on the quantities appearing in parentheses to clarify the dependence.
%

\section{Modified energies for approximate problems}
\label{sec:2}
The key for the proof of Theorem \ref{thm:1.1} is to compute suitable energies for $H^2$ solutions 
to \eqref{eq:1.1}. To justify this procedure, we need to consider approximate problems because higher-order derivatives appear in the intermediate computations. Our aim in this section is to derive the $H^2$ identity for the approximate equation.

\subsection{Approximate equation}
According to \cite{AS15}, we introduce the cutoff operator in Fourier space as
\begin{align*}
(J_\eps f)(x)\ce 
\sum_{\substack{n\in\Z\\|n|\leq1/\eps}}\hat{f}(n) e^{2\pi i nx},
\quad x\in\T
\end{align*}
for $\eps\in(0,1)$. 
The basic properties of $J_\eps$ can be summarized as follows.
\begin{lemma}
\label{lem:2.1}
For $\eps\in(0,1)$, $f,g\in L^2$, and $s>0$, the following properties hold: 
\begin{enumerate}
\item $\ds J_\eps^2=J_\eps$,
\item $\rbra[J_\eps f,g]_{L^2}=\rbra[f,J_\eps g]_{L^2}$,
\item $\norm[J_\eps f]_{L^2}\leq\norm[f]_{L^2}$,
\item $\norm[J_\eps f]_{H^s}\cleq \eps^{-s}\norm[f]_{L^2}$,
\item $\norm[J_\eps f-f]_{H^s}\to 0$ as $\eps\downarrow0$ for any $f\in H^s$.
\end{enumerate}
\end{lemma}

We consider the approximate equation for \eqref{eq:1.1}:
\begin{align}
\label{eq:2.1}
\l\{
\begin{aligned}
&i\pt_t u_\eps+\pt_x^2 u_\eps+iJ_\eps\l(|J_\eps u_\eps|^{2\sigma}\pt_xJ_\eps u_\eps\r)=0,
\\
&{u_\eps}_{\mid t=0}=J_\eps\varphi ,
\end{aligned}
\r.
\quad (t,x)\in\R\times\T,
\end{align}
where $\eps\in(0,1)$. The existence and uniqueness for this approximate equation is easily obtained by the standard argument.
\begin{lemma}
\label{lem:2.2}
Let $\eps\in(0,1)$. For any $\varphi\in L^2$ there exists a unique solution $u_\eps\in C(\R, L^2)$ to \eqref{eq:2.1}. Moreover, $u_\eps \in C^1(\R,H^\infty)$ and $\norm[u_\eps(t)]_{L^2}^2=\norm[u_\eps(0)]_{L^2}^2$ for all $t\in\R$.
\end{lemma} 
\begin{proof}
For completeness we give a proof. Similar arguments are done in the proof of \cite[Theorem 3.3.1]{C03}.
We set 
\begin{align}
\label{eq:2.2}
g(u)=i|u|^{2\sigma}\pt_xu,\quad g_\eps(u)=J_\eps g(J_\eps u).
\end{align}
Note that $g_\eps$ is Lipschitz continuous on bounded subsets of $L^2$ for a fixed $\eps\in(0,1)$. By a fixed point theorem, one can prove that for any $\varphi\in L^2$, there exists a unique maximal solution $u_\eps\in C( (-T_1,T_2), L^2)$ with $T_1,T_2\in(0,\infty]$, and if $T_1<\infty$, then $\norm[u_\eps(t)]_{L^2}\to\infty$ as $t\downarrow -T_1$ (respectively, if $T_2<\infty$, then $\norm[u_\eps(t)]_{L^2}\to\infty$ as $t\uparrow T_2$). 
\\
By Duhamel's formula, $u_\eps$ satisfies
\begin{align*}
u_\eps(t)=U(t)J_\eps\varphi+i\int_0^t U(t-s)J_\eps g(J_\eps u_\eps(s))ds
\end{align*} 
for $t\in(-T_1,T_2)$, where $U(t)=e^{it\pt_x^2}$. Then, we obtain from the property of $J_\eps$ that $u_\eps\in C( (-T_1,T_2), H^\infty)$. From the equation \eqref{eq:2.1} we obtain $\pt_t u_\eps\in C( (-T_1,T_2), H^\infty)$, which implies that $u\in C^1( (-T_1,T_2), H^\infty)$.
\\
We note that $g_\eps$ satisfies
\begin{align*}
\Im\int g_\eps(u)\widebar{u}dx=\Im\int g(J_\eps u)\widebar{J_\eps u}dx=0
\end{align*}
for any $u\in L^2$.
From this property and the equation \eqref{eq:2.1}, we obtain the conservation of the $L^2$ norm
\begin{align}
\label{eq:2.3}
\norm[u_\eps(t)]_{L^2}^2=\norm[u_\eps(0)]_{L^2}^2=\norm[J_\eps\varphi]_{L^2}^2
\end{align}
for all $t\in(-T_1,T_2)$, which implies that $T_1=T_2=\infty$.
\end{proof}


\subsection{Modified energies for approximate equations}
\label{sec:2.2}
We first introduce the following terminology.
\begin{definition}
\label{def:2.3}
We define the set $\scG$ of all functionals $G\in C(H^2,\R)$ such that for all $M>0$ there exists $C(M)>0$ such that
\begin{align*}
u\in H^2,~\norm[u]_{H^1}\leq M \implies \abs[G(u)]\leq C(M)(1+\norm[u]_{H^2}^2).
\end{align*}
If $G\in\scG$, then we call the value of the functional $G(u)$ a \emph{good term}. For a time-dependent function $u\in C(\R, H^2)$, we may call $G(u)$ a good term in the sense that $G(u(t))$ is a good term for each $t\in\R$. Roughly speaking, a good term here means a term that does not cause any harm when one derives a priori estimates on $H^2$ by using Gronwall's lemma. 
\end{definition}
Let $\eps\in(0,1)$ and let $u_\eps$ be a unique smooth solution of \eqref{eq:2.1}. We set $v_\eps =J_\eps u_\eps$. The main result in this section is the following.
\begin{theorem}%
\label{thm:2.4}
There exists $G\in\scG$ such that the following identity holds:
\begin{align}
\label{eq:2.4}
\frac{d}{dt}
\begin{aligned}[t]
&\Bigl( \norm[\pt^2u_\eps]_{L^2}^2
-2\Im\int \pt^2\widebar{v}_\eps\pt v_\eps\abs[v_\eps]^{2\sigma} 
-\frac{2\sigma}{\sigma+1}\Im\int\pt^2v_\eps\pt v_\eps\widebar{v}_\eps^{2}\abs[v_\eps]^{2(\sigma-1)}
\\
&{}\qquad+\frac{\sigma(\sigma-1)}{2(\sigma+1)}\Im\int(\pt v_\eps)^3\widebar{v}_\eps^{3}\abs[v_\eps]^{2(\sigma-2)}\Bigr)
=G(v_\eps)
\end{aligned}
\end{align}
for all $t\in\R$. 
\end{theorem}
The heuristic idea behind the construction of the energy introduced along Theorem \ref{thm:2.4} is explained in Section \ref{sec:4}.
For the rest of this section we will prove Theorem \ref{thm:2.4}.
We rewrite \eqref{eq:2.1} as
\begin{align}
\label{eq:2.5}
\pt_t u_\eps=i\pt^2 u_\eps-J_\eps\l( |J_\eps u_\eps|^{2\sigma}\pt J_\eps u_\eps\r),\quad (t,x)\in\R\times\T.
\end{align}
We begin with the following lemma.
\begin{lemma}
\label{lem:2.5}
The functionals
\begin{align*}
I_k(u)=\Re\int\pt^2\widebar{u}(\pt u)^{3-k}(\pt\widebar{u})^k|u|^{2(\sigma-2)}u^k\widebar{u}^{2-k},
\quad k\in\{0,1,2\},
\end{align*}
which are well-defined on $H^2$, satisfy the property $I_k\in\scG$.
\end{lemma}
\begin{proof}
Applying the Gagliardo--Nirenberg inequality
\begin{align*}
\norm[f]_{L^6}^6 \cleq \norm[f]_{H^1}^2 \norm[f]_{L^2}^4
\end{align*}
and $H^1\subset L^\infty$, one can estimate
\begin{align*}
\abs[I_k(u)]
&\leq
\int \abs[\pt^2u] \abs[\pt u]^3\abs[u]^{2(\sigma-1)}
\\
&\cleq \norm[\pt^2u]_{L^2}\norm[\pt u]_{L^6}^3\norm[u]_{L^\infty}^{2(\sigma-1)} \cleq \norm[u]_{H^2}^2 \norm[u]_{H^1}^{2\sigma}.
\end{align*}
This implies $I_k\in\scG$.
\end{proof}
We now start to calculate the $H^2$ energy for the approximate equation \eqref{eq:2.1}. We define the functionals $B_1$ and $B_2$ by
\begin{align}
\label{eq:2.6}
B_1(u)&=\int|\pt^2u|^2\pt(|u|^{2\sigma})=\sigma\int|\pt^2u|^2|u|^{2(\sigma-1)}\pt(|u|^2),
\\
\label{eq:2.7}
B_2(u)&=\sigma\Re\int(\pt^2\widebar{u})^2\pt uu|u|^{2(\sigma-1)}
=\sigma\Re\int(\pt^2u)^2{\pt\widebar{u}}\widebar{u}|u|^{2(\sigma-1)}.
\end{align}
\begin{lemma}
\label{lem:2.6}
There exists $G_0\in\scG$ such that
\begin{align}
\label{eq:2.8}
\frac{d}{dt}\norm[\pt^2 u_\eps]_{L^2}^2= {-}4B_1(v_\eps)-2B_2(v_\eps)+G_0(v_\eps)
\end{align}
for all $t\in\R$.
\end{lemma}
\begin{proof}
A direct calculation shows that
\begin{align*}
\frac{d}{dt}\norm[\pt^2 u_\eps]_{L^2}^2
&=2\Re\rbra[\pt_t\pt^2u_\eps,\pt^2u_\eps]_{L^2}
\\
&=2\Re\int i\pt^4u_\eps\pt^2\widebar{u}_\eps-2\Re\int\pt^2(|J_\eps u_\eps|^{2\sigma}\pt J_\eps u_\eps)\pt^2\widebar{J_\eps u_\eps}
\\
&=-2\Re\int\pt^2(|v_\eps|^{2\sigma})\pt v_\eps\pt^2\widebar{v}_\eps
-4\int\pt(|v_\eps|^{2\sigma})|\pt^2v_\eps|^2-2\Re\int|v_\eps|^{2\sigma}\pt^3v_\eps\pt^2\widebar{v}_\eps.
\end{align*}
The term $\pt^2(|v_\eps|^{2\sigma})$ is represented by a linear combination of the five terms
\begin{align}
\label{eq:2.9}
\begin{aligned}
&\pt^2v_\eps\widebar{v}_\eps\abs[v_\eps]^{2(\sigma-1)},~\pt^2\widebar{v}_\eps v_\eps\abs[v_\eps]^{2(\sigma-1)},
\\
&(\pt v_\eps)^2\widebar{v}_\eps^{2}\abs[v_\eps]^{2(\sigma-2)},~|\pt v_\eps|^2\abs[v_\eps]^{2(\sigma-1)},~
(\pt\widebar{v}_\eps)^2v_\eps^2\abs[v_\eps]^{2(\sigma-2)}.
\end{aligned}
\end{align}
This relation may be justified by the calculation $\pt^2(|v_\eps|^2+\delta)^\sigma$ and passing to the limit $\delta\downarrow0$. The last three terms in \eqref{eq:2.9} correspond to $I_k(v_\eps)$ in Lemma \ref{lem:2.5} respectively, and they can be treated as good terms. 
Therefore, there exists $G_0\in\scG$ such that
\begin{align*}
\frac{d}{dt}\norm[\pt^2 u_\eps]_{L^2}^2
&= 
\begin{aligned}[t]
&{-}2\sigma\Re\int(\pt^2v_\eps\widebar{v}_\eps\abs[v_\eps]^{2(\sigma-1)}+\pt^2\widebar{v}_\eps v_\eps\abs[v_\eps]^{2(\sigma-1)})
\pt v_\eps\pt^2\widebar{v}_\eps
\\
&\quad 
{-}4\int\pt(|v_\eps|^{2\sigma})|\pt^2v_\eps|^2 
-\int|v_\eps|^{2\sigma}\pt|\pt^2v_\eps|^2+G_0(v_\eps)
\end{aligned}
\\
&={-}\sigma\int|\pt^2v_\eps|^2|v_\eps|^{2(\sigma-1)}\pt(|v_\eps|^2)
-2B_2(v_\eps)-3B_1(v_\eps)+G_0(v_\eps)
\\
&=-4B_1(v_\eps)-2B_2(v_\eps)+G_0(v_\eps),
\end{align*}
which proves \eqref{eq:2.8}.
\end{proof}

Next, we calculate the time derivative of the correction terms on the LHS of \eqref{eq:2.4}. The first correction term is calculated as follows.
\begin{lemma}
\label{lem:2.7}
There exists $G_1\in\scG$ such that
\begin{align}
\label{eq:2.10}
\frac{d}{dt}\Im\int \pt^2\widebar{v}_\eps\pt v_\eps|v_\eps|^{2\sigma}=-2B_1(v_\eps)-2B_2(v_\eps)+G_1(v_\eps)
\end{align}
for all $t\in\R$.
\end{lemma}
\begin{proof}
A direct calculation shows that the LHS of \eqref{eq:2.10} equals
\begin{align}
\label{eq:2.11}
\begin{aligned}
&\Im\int\pt_t\pt^2\widebar{v}_\eps\pt v_\eps|v_\eps|^{2\sigma}
+\Im\int\pt^2\widebar{v}_\eps\pt_t\pt v_\eps|v_\eps|^{2\sigma}
\\
&\quad
+\sigma\Im\int\pt^2\widebar{v}_\eps\pt v_\eps
\abs[v_\eps]^{2(\sigma-1)}\l(\pt_tv_\eps 
\widebar{v}_\eps+v_\eps\pt_t\widebar{v}_\eps \r).
\end{aligned}
\end{align}
We now rewrite the time derivative in \eqref{eq:2.11} by using the equation \eqref{eq:2.5}. 

We note that the replacement $\pt_t v_\eps\rightarrow$ the nonlinearity can be treated as good terms as follows. 
For the first term of \eqref{eq:2.11}, this replacement gives
\begin{align*}
{-}\Im\int\pt^2J_\eps(|v_\eps|^{2\sigma}\pt\widebar{v}_\eps)J_\eps(|v_\eps|^{2\sigma}\pt v_\eps)
=\Im\int J_\eps\l[\pt(|v_\eps|^{2\sigma}\pt\widebar{v}_\eps)\r] J_\eps\l[\pt(|v_\eps|^{2\sigma}\pt v_\eps)\r]=0.
\end{align*}
The same replacement for the second term of \eqref{eq:2.11} is estimated as
\begin{align}
\label{eq:2.12}
\abs[ {\Im\int
\pt^2\widebar{v}_\eps J_\eps^2\l[\pt(|v_\eps|^{2\sigma}\pt\widebar{v}_\eps)\r] |v_\eps|^{2\sigma} } ]
\cleq 
\norm[\pt^2 v_\eps]_{L^2}^2\norm[v_\eps]_{L^\infty}^{4\sigma}
+\norm[\pt^2 v_\eps]_{L^2}\norm[\pt v_\eps]_{L^4}^2
\norm[v_\eps]_{L^\infty}^{4\sigma-1},
\end{align}
which implies that this replacement gives a good term. The third term of \eqref{eq:2.11} can be treated similarly by the same replacement.

Therefore, there exists $G_{11}\in\scG$ such that \eqref{eq:2.11} equals
\begin{align*}
&
\begin{aligned}[t]
&{-}\Re\int\pt^4\widebar{v}_\eps\pt v_\eps|v_\eps|^{2\sigma}+\frac{1}{2}\int\pt(|\pt^2v_\eps|)|v_\eps|^{2\sigma}
\\
&\quad
+\frac{1}{2}\int|\pt^2v_\eps|^2{\sigma|v_\eps|^{2(\sigma-1)}\pt(|v_\eps|^2)}
-\sigma\Re\int(\pt^2\widebar{v}_\eps)^2\pt v_\eps v_\eps|v_\eps|^{2(\sigma-1)}+G_{11}(v_\eps)
\end{aligned}
\\
&={-}\Re\int\pt^2\widebar{v}_\eps\pt^2(\pt v_\eps|v_\eps|^{2\sigma})
-B_2(v_\eps)+ G_{11}(v_\eps).
\end{align*}
The first term on the RHS of the last equality equals
\begin{align*}
&{-}\frac{1}{2}\int\pt(|\pt^2v_\eps|^2)|v_\eps|^{2\sigma}-2\int|\pt^2v_\eps|^2\pt(|v_\eps|^{2\sigma})
-\Re\int\pt^2\widebar{v}_\eps\pt v_\eps\pt^2(|v_\eps|^{2\sigma}).
\end{align*}
By recalling the calculation of $\pt^2(|v_\eps|^{2\sigma})$ in the proof of Lemma \ref{lem:2.6}, we deduce that there exists $G_{12}\in\scG$ such that the previous formula equals
\begin{align*}
&{-}\frac{3}{2}B_1(v_\eps)-\sigma\Re\int\pt^2\widebar{v}_\eps\pt v_\eps|v_\eps|^{2(\sigma-1)}(\pt^2 v_\eps\widebar{v}_\eps+v_\eps\pt^2\widebar{v}_\eps)+G_{12}(v_\eps)
\\
={}&{-}\frac{3}{2}B_1(v_\eps)-\frac{1}{2}\sigma\int|\pt^2v_\eps|^2|v_\eps|^{2(\sigma-1)}\pt(|v_\eps|^2)
-\sigma\Re\int(\pt^2\widebar{v}_\eps)^2\pt v_\eps v_\eps|v_\eps|^{2(\sigma-1)}+G_{12}(v_\eps)
\\
={}&{-}2B_1(v_\eps)-B_2(v_\eps)+G_{12}(v_\eps).
\end{align*}
Hence we conclude \eqref{eq:2.10} by setting $G_1=G_{11}+G_{12}$.
\end{proof}
The calculation of the second correction term on the LHS of \eqref{eq:2.4} is a little more complicated. We define the functional $B_3$ by
\begin{align}
\label{eq:2.13}
B_3(u)&= \sigma(\sigma-1)\Re\int(\pt^2u)^2\pt u \widebar{u}^{3}|u|^{2(\sigma-2)}.
\end{align}
\begin{lemma}
\label{lem:2.8}
There exists $G_2\in\scG$ such that
\begin{align}
\label{eq:2.14}
\frac{d}{dt}\,\sigma\Im\int\pt^2v_\eps\pt v_\eps\widebar{v}_\eps^{2}|v_\eps|^{2(\sigma-1)}
= (\sigma+1)B_2(v_\eps) +3B_3(v_\eps)+G_2(v_\eps)
\end{align}
for all $t\in\R$.
\end{lemma}
\begin{proof}
A direct calculation shows that the LHS of \eqref{eq:2.14} equals
\begin{align*}
\begin{aligned}[t]
&\sigma\Im\int\pt_t\pt^2v_\eps\pt v_\eps\widebar{v}_\eps^{2}|v_\eps|^{2(\sigma-1)}
+\sigma\Im\int\pt^2v_\eps\pt_t\pt v_\eps\widebar{v}_\eps^{2}|v_\eps|^{2(\sigma-1)}
\\
&\quad
+\sigma(\sigma-1)\Im\int\pt^2v_\eps\pt v_\eps\pt_tv_\eps\widebar{v}_\eps^3|v_\eps|^{2(\sigma-2)}
+\sigma(\sigma+1)\Im\int\pt^2v_\eps\pt v_\eps 
\widebar{v}_\eps\pt_t\widebar{v}_\eps|v_\eps|^{2(\sigma-1)}.
\end{aligned}
\end{align*}
Similar to the proof of Lemma \ref{lem:2.7}, the replacement $\pt_tv_\eps\rightarrow$ the nonlinearity yields good terms. 
Therefore, there exists $G_{21}\in\scG$ such that the previous formula equals
\begin{align*}
&
\begin{aligned}[t]
&\sigma\Re\int\pt^4v_\eps\pt v_\eps\widebar{v}_\eps^{2}|v_\eps|^{2(\sigma-1)}
+\sigma\Re\int\pt^2v_\eps\pt^3 v_\eps\widebar{v}_\eps^{2}|v_\eps|^{2(\sigma-1)}
\\
&\quad+\sigma(\sigma-1)\Re\int(\pt^2v_\eps)^2\pt v_\eps\widebar{v}_\eps^3|v_\eps|^{2(\sigma-2)}
-\sigma(\sigma+1)\Re\int\abs[\pt^2v_\eps]^2\pt v_\eps 
\widebar{v}_\eps|v_\eps|^{2(\sigma-1)}
+G_{21}(v_\eps)
\end{aligned}
\\
={}&
\begin{aligned}[t]
&\sigma\Re\int\pt^4v_\eps\pt v_\eps\widebar{v}_\eps^{2}|v_\eps|^{2(\sigma-1)}
+\frac{\sigma}{2}\Re\int\pt\l( (\pt^2v_\eps)^2 \r)\widebar{v}_\eps^{2}|v_\eps|^{2(\sigma-1)}
\\
&\quad +B_3(v_\eps)-\frac{\sigma+1}{2}B_1(v_\eps)+G_{21}(v_\eps).
\end{aligned}
\end{align*}
We now calculate the first two terms on the RHS of the last equality. By integration by parts, the first term equals
\begin{align*}
&
\begin{aligned}[t]
&{-}\sigma\Re\int\pt^3v_\eps\pt^2 v_\eps\widebar{v}_\eps^{2}|v_\eps|^{2(\sigma-1)}
-\sigma(\sigma+1)\Re\int\pt^3v_\eps|\pt v_\eps|^2 \widebar{v}_\eps|v_\eps|^{2(\sigma-1)}
\\
&\quad
-\sigma(\sigma-1)\Re\int
\pt^3v_\eps(\pt v_\eps)^2\widebar{v}_\eps^{3}|v_\eps|^{2(\sigma-2)}
\end{aligned}
\\
={}&
\begin{aligned}[t]
&\frac{\sigma}{2}\Re\int(\pt^2 v_\eps)^2\pt\l(\widebar{v}_\eps^{2}|v_\eps|^{2(\sigma-1)}\r)
+\sigma(\sigma+1)\Re\int(\pt^2v_\eps)^2\pt\widebar{v}_\eps \widebar{v}_\eps|v_\eps|^{2(\sigma-1)}
\\
&{}
+\sigma(\sigma+1)\Re\int|\pt^2v_\eps|^2\pt v_\eps \widebar{v}_\eps|v_\eps|^{2(\sigma-1)}
+\sigma(\sigma+1)\Re\int \pt^2v_\eps|\pt v_\eps|^2\pt\l(\widebar{v}_\eps|v_\eps|^{2(\sigma-1)}\r)
\\
&{}+2\sigma(\sigma-1)\Re\int
(\pt^2v_\eps)^2\pt v_\eps\widebar{v}_\eps^{3}|v_\eps|^{2(\sigma-2)}
+\sigma(\sigma-1)\Re\int
\pt^2v_\eps(\pt v_\eps)^2\pt\l(\widebar{v}_\eps^{3}|v_\eps|^{2(\sigma-2)}\r).
\end{aligned}
\end{align*}
By Lemma \ref{lem:2.5}, one can see that the fourth term and the sixth term on the RHS of the last equality are good terms. Thus, there exists $G_{22}\in\scG$ such that the previous formula equals
\begin{align*}
&
\begin{aligned}[t]
&\frac{\sigma(\sigma-1)}{2}
\Re\int(\pt^2 v_\eps)^2\pt v_\eps\widebar{v}_\eps^{3}
|v_\eps|^{2(\sigma-2)}
+\frac{\sigma(\sigma+1)}{2}
\Re\int(\pt^2 v_\eps)^2\pt \widebar{v}_\eps \widebar{v}_\eps
|v_\eps|^{2(\sigma-1)}
\\
&{}~+(\sigma+1)B_2(v_\eps)+\frac{\sigma+1}{2}B_1(v_\eps)
+2B_3(v_\eps)+G_{22}(v_\eps)
\end{aligned}
\\
={}&\frac{\sigma+1}{2}B_1(v_\eps)+\frac{3(\sigma+1)}{2}B_2(v_\eps)+\frac{5}{2}B_3(v_\eps)+G_{22}(v_\eps).
\end{align*}
Similarly, from integration by parts we obtain
\begin{align*}
&\frac{\sigma}{2}\Re\int\pt\l( (\pt^2v_\eps)^2 \r)\widebar{v}_\eps^{2}|v_\eps|^{2(\sigma-1)}={-}\frac{\sigma}{2}\Re\int(\pt^2v_\eps)^2\pt\l(\widebar{v}_\eps^{2}|v_\eps|^{2(\sigma-1)}\r)
\\
={}&-\frac{\sigma(\sigma-1)}{2}\Re\int(\pt^2v_\eps)^2\pt v_\eps \widebar{v}_\eps^{3}|v_\eps|^{2(\sigma-2)}
-\frac{\sigma(\sigma+1)}{2}
\Re\int(\pt^2v_\eps)^2\pt\widebar{v}_\eps\widebar{v}_\eps|v_\eps|^{2(\sigma-1)}
\\
={}&-\frac{1}{2}B_3(v_\eps) -\frac{\sigma+1}{2}B_2(v_\eps).
\end{align*}
Collecting these calculations, we obtain \eqref{eq:2.14} by setting $G_2=G_{21}+G_{22}$.
\end{proof}

Finally, the third correction term on the LHS of \eqref{eq:2.4} is calculated as follows.
\begin{lemma}
\label{lem:2.9}
There exists $G_3\in\scG$ such that
\begin{align}
\label{eq:2.15}
\frac{d}{dt}\sigma(\sigma-1)
\Im\int(\pt v_\eps)^3\widebar{v}_\eps^3|v_\eps|^{2(\sigma-2)}
= -6B_3(v_\eps)+G_3(v_\eps).
\end{align}
for all $t\in\R$.
\end{lemma}
\begin{proof}
The LHS of \eqref{eq:2.15} can be computed as follows
\begin{align}
\label{eq:2.16}
3\sigma(\sigma-1) \Im\int(\pt v_\eps)^2\pt_t\pt v_\eps
\widebar{v}_\eps^3|v_\eps|^{2(\sigma-2)}
+\sigma(\sigma-1)\Im\int(\pt v_\eps)^3\pt_t(\widebar{v}_\eps^3|v_\eps|^{2(\sigma-2)}).
\end{align}
Regarding the second term in \eqref{eq:2.16}, we note that
\begin{align*}
\pt_t(\widebar{v}_\eps^3|v_\eps|^{2(\sigma-2)})
=(\sigma-2)\pt_tv_\eps\widebar{v}_\eps^4|v_\eps|^{2(\sigma-3)}+(\sigma+1)\pt_t\widebar{v}_\eps\widebar{v}_\eps^2|v_\eps|^{2(\sigma-2)},
\end{align*}
which makes sense when $\sigma>1$. Therefore, when we rewrite the time derivative by the equation \eqref{eq:2.5}, the second term in \eqref{eq:2.16} is expressed as $G_{31}(v_\eps)$ for some $G_{31}\in\scG$.

Regarding the first term in \eqref{eq:2.16}, similarly to the proof of Lemma \ref{lem:2.7}, the replacement $\pt_tv_\eps\rightarrow$ the nonlinearity is expressed as $G_{32}(v_\eps)$ for some $G_{32}\in\scG$. Thus, by integration by parts the first term in \eqref{eq:2.16} equals
\begin{align*}
&3\sigma(\sigma-1) \Re\int(\pt v_\eps)^2\pt^3v_\eps \widebar{v}_\eps^3|v_\eps|^{2(\sigma-2)}+G_{32}(v_\eps)
\\
={}&{-}6\sigma(\sigma-1)\int (\pt^2v_\eps)^2\pt v_\eps\widebar{v}_\eps^3|v_\eps|^{2(\sigma-2)}+G_{32}(v_\eps)+G_{33}(v_\eps)
\\
={}&{-}6B_3(v_3)+G_{32}(v_\eps)+G_{33}(v_\eps),
\end{align*}
where we have set
\begin{align*}
G_{33}(v_\eps)={-}3\sigma(\sigma-1) \Re\int(\pt v_\eps)^2\pt^2v_\eps 
\pt\l(\widebar{v}_\eps^3|v_\eps|^{2(\sigma-2)}\r),
\end{align*}
which is a good term. Hence \eqref{eq:2.16} follows by setting $G_3=G_{31}+G_{32}+G_{33}$.
\end{proof}
\begin{proof}[Proof of Theorem \ref{thm:2.4}]
The conclusion follows from Lemmas \ref{lem:2.6}, \ref{lem:2.7}, \ref{lem:2.8}, and \ref{lem:2.9}. Indeed each coefficient of the energy at the LHS in \eqref{eq:2.4} is set to cancel out $B_1(v_\eps)$, $B_2(v_\eps)$, and $B_3(v_\eps)$ (see also the discussion in Section \ref{sec:4}). 
\end{proof}

\section{Global existence of $H^2$ solutions}
\label{sec:3}
In this section we prove Theorem \ref{thm:1.1} based on the $H^2$ identity \eqref{eq:2.4} for approximate solutions. For simplicity we only consider the positive time direction.

\subsection{Convergence of approximate solutions}
We recall the approximate equation introduced in Section \ref{sec:2}:
\begin{equation}\label{eq:3.1}
i\pt_t u_\eps+\pt_x^2 u_\eps+iJ_\eps\l(|J_\eps u_\eps|^{2\sigma}\pt_xJ_\eps u_\eps\r)=0,
\quad (t,x)\in\R\times\T.
\end{equation}
We have the following claim about the convergence of approximate solutions.
\begin{lemma}
\label{lem:3.1}
Let $\varphi\in H^2$ and let $u_\eps$ be the smooth solution of \eqref{eq:3.1} with $u_\varepsilon(0)=J_\varepsilon \varphi$. Assume that for a given $T>0$,
\begin{align}
\label{eq:3.2}
\sup_{\eps\in(0,1)}\sup_{t\in[0,T]}\norm[u_\eps(t)]_{H^2}<\infty.
\end{align}
Then, there exists $u\in C([0,T], H^2)$ such that
\begin{align}
\label{eq:3.3}
u_\eps(t)&\wto u(t)\quad\text{in}~H^2
\quad\text{for all}~t\in[0,T],
\\
\label{eq:3.4}
u_\eps&\to u\quad\text{in}~C([0,T], H^s)~\text{with $s<2$},
\end{align}
and $u$ gives a unique $H^2$ solution to \eqref{eq:1.1}.
\end{lemma}
\begin{proof}
Following the argument of \cite[Section 2.2]{HO16}, one can prove that $\{u_\eps\}_{0<\eps<1}$ forms a Cauchy sequence in $C([0,T], L^2)$.\footnote{In \cite{HO16}, the operator $(I-\eps\pt^2_x)^{-1}$ is used instead of $J_\eps$ in approximate equations, but this difference does not affect the argument.} Combining this with \eqref{eq:3.2}, one can prove that there exists 
\begin{align*}
u\in C_w([0,T], H^2)\cap 
\bigcap_{s\in[0,2)}C([0,T], H^s)
\end{align*}
and the convergences \eqref{eq:3.3} and \eqref{eq:3.4} hold. We remark that the weak convergence \eqref{eq:3.3} can be obtained independent of weak compactness (see \cite[Lemma 2.5]{HOpre} for more details). 

By \eqref{eq:3.3} and \eqref{eq:3.4} one can easily prove that $u$ is the $H^2$ solution of \eqref{eq:1.1} (see \cite[Section 2.3]{HO16} for details). To show $u\in C([0,T],H^2)$, we use the argument of \cite[Remarks (c)]{KL84}, which is actually used in \cite[Section 4]{AS15} for \eqref{eq:1.1}. 
We briefly explain it here. 
First it follows from the weak continuity of $t\mapsto u(t)\in H^2$ that
\begin{align}
\label{eq:3.5}
\norm[u(0)]_{H^2}^2\leq\liminf_{t\to0}\norm[u(t)]_{H^2}^2.
\end{align}
Next, we note that
\begin{align}
\label{eq:3.6}
\frac{d}{dt} \norm[u_\eps(t)]_{H^2}^2\cleq \norm[u_\eps(t)]_{L^\infty}^{2\sigma-1}\norm[\pt u_\eps(t)]_{L^\infty}\norm[u_\eps(t)]_{H^2}^2
\cleq \norm[u_\eps(t)]_{H^2}^{2\sigma+2},
\end{align}
which is easily obtained by using the equation \eqref{eq:2.1} and Sobolev's embedding (see \cite[Lemma 4.1]{AS15}). 
From \eqref{eq:3.6} and \eqref{eq:3.3} one can prove that
\begin{align}
\label{eq:3.7}
\limsup_{t\to0}\norm[u(t)]_{H^2}^2\leq\norm[u(0)]_{H^2}^2.
\end{align}
Therefore, it follows from \eqref{eq:3.5} and \eqref{eq:3.7} that the strong continuity of $t\mapsto u(t)\in H^2$ holds at $t=0$. This argument does not depend on the initial time and hence we deduce that $u\in C([0,T], H^2)$.
\end{proof}

\subsection{Proof of the theorem}
\label{sec:3.2}

We prove Theorem \ref{thm:1.1} by contradiction. For the maximal $H^2$ solution $u\in C([0,T_{\rm max}), H^2)$ to \eqref{eq:1.1}, we assume that
\begin{align}
\label{eq:3.8}
T_{\rm max}<\infty\quad\text{and}\quad
\sup_{t\in[0,T_{\rm max})}\norm[u(t)]_{H^1}<\infty. 
\end{align}
We need the following result in order to apply Theorem \ref{thm:2.4}. 
Its proof is given in Section \ref{sec:3.4} below.
\begin{lemma}
\label{lem:3.2}
Assume \eqref{eq:3.8} for the maximal $H^2$ solution of \eqref{eq:1.1}. Then, there exists $M_*>0$ such that the following holds: For any $T\in(0,T_{\rm max})$ there exists $\eps_*\in (0,1)$ such that 
\begin{align}
\label{eq:3.9}
\sup_{\eps\in(0,\eps_*)}\sup_{t\in[0,T]}\norm[u_\eps(t)]_{H^1}\leq M_*.
\end{align}
\end{lemma}
We now complete the proof of Theorem \ref{thm:1.1} assuming Lemma \ref{lem:3.2}.
\begin{proof}[Proof of Theorem \ref{thm:1.1}]
Take any $T\in (0,T_{\rm max})$. Once Lemma \ref{lem:3.2} is established, then we can conclude as follows by using the identity \eqref{eq:2.4}.
 After integrating in time \eqref{eq:2.4} and elementary considerations,
there exist constants $C_1, C_2>0$ depending only on $M_*$ (neither on $T$ nor $\eps$) such that  \begin{align}
\label{eq:3.10}
\norm[u_\eps(t)]_{H^2}^2\leq C_1(1+\norm[\varphi]_{H^2}^{2})+C_2\int_0^t(1+\norm[u_\eps(\tau)]_{H^2}^2)d\tau
\end{align}
for $\eps\in(0,\eps_*)$ and $t\in[0,T]$. Therefore, by Gronwall's lemma we deduce that
\begin{align}
\label{eq:3.11}
\norm[u_\eps(t)]_{H^2}^2\leq C_1(1+\norm[\varphi]_{H^2}^{2}) e^{C_2t},
\end{align}
which in particular implies that 
\begin{align*}
\sup_{\eps\in(0,\eps_*)}\sup_{t\in[0,T]}\norm[u_\eps(t)]_{H^2(\T)}^2
\leq C_1(1+\norm[\varphi]_{H^2}^{2}) e^{C_2T}.
\end{align*}
Therefore, it follows from \eqref{eq:3.11} and Lemma \ref{lem:3.1} that
\begin{align*}
\sup_{t\in [0, T]} \norm[u(t)]_{H^2}^2\leq C_1(1+\norm[\varphi]_{H^2}^{2})e^{C_2T}
\end{align*} 
which implies a contradiction in the case $T_{\rm max}<\infty$.
\end{proof}
The remaining of this section is devoted to the proof of Lemma \ref{lem:3.2}. 
\subsection{Small data case} 
\label{sec:3.3}

If we assume the $H^1$ smallness of the initial data, the proof of Lemma \ref{lem:3.2} becomes simpler. 
We define the energy of \eqref{eq:2.1} by
\begin{align*}
E_\eps(u)=\frac{1}{2}\int |\pt u|^2 -\frac{1}{2\sigma+2}
\Re\int |J_\eps u|^{2\sigma}\pt J_\eps u\widebar{J_\eps u}
\quad \text{for}~\eps\in(0,1).
\end{align*}
It is easily verified that the solution $u_\eps$ constructed by Lemma \ref{lem:2.2} satisfies the conservation law of the energy
\begin{align*}
E_\eps (u_\eps (t))= E_\eps(u_\eps(0))=E_\eps(J_\eps\varphi)
\end{align*}
for all $t\in\R$. Based on the conservation laws of the $L^2$ norm and the energy, we deduce that for all $t\in\R$,
\begin{align*}
\frac{1}{2}\norm[u_\eps(t)]_{H^1}^2
&=\frac{1}{2}\norm[u_\eps(t)]_{L^2}^2
+E_\eps(u_\eps(t))+\frac 1{2\sigma+2} \Re\int|v_\eps(t)|^{2\sigma}\pt v_\eps(t)\widebar{v_\eps(t)}
\\
&\leq \frac{1}{2}\norm[J_\eps\varphi]_{L^2}^2
+E_\eps(J_\eps\varphi)+\frac{c}{2\sigma +2}\norm[v_\eps(t)]_{H^1}^{2\sigma+2}
\\
&\leq
\frac{1}{2}\norm[\varphi]_{H^1}^2+\frac{c}{2\sigma+2}\norm[\varphi]_{H^1}^{2\sigma+2}+\frac{c}{2\sigma +2}\norm[u_\eps(t)]_{H^1}^{2\sigma+2},
\end{align*}
where $v_\eps=J_\eps u_\eps$ and $c$ is a positive constant. Therefore, we obtain the relation
\begin{align}
\label{eq:3.12}
h(\norm[u_\eps(t)]_{H^1})\leq \frac{1}{2}\norm[\varphi]_{H^1}^2+\frac{c}{2\sigma+2}\norm[\varphi]_{H^1}^{2\sigma+2}
\quad\text{for all}~t\in\R,
\end{align}
where the function $h:[0,\infty)\to\R$ is defined by $h(s)=s^2/2-c/(2\sigma+2)s^{2\sigma+2}$. We note that $h$ has a unique maximum point $m\ce (1/c)^{1/(2\sigma)}$. We take $\delta>0$ small enough so that
\begin{align*}
\norm[\varphi]_{H^1}<\delta\implies
\frac{1}{2}\norm[\varphi]_{H^1}^2+\frac{c}{2\sigma+2}\norm[\varphi]_{H^1}^{2\sigma+2}<h(m).
\end{align*}
Therefore, it follows from \eqref{eq:3.12} and the strong continuity $t\mapsto u_\eps(t)\in H^1(\T)$ that if $\norm[\varphi]_{H^1}<\delta$, then $\norm[u_\eps(t)]_{H^1}<m$ for all $t\in\R$.  Hence, we have
\begin{align*}
\sup_{\eps\in(0,1)}\sup_{t\in\R}\norm[u_\eps(t)]_{H^1}\leq m,
\end{align*}
which in particular implies the conclusion of Lemma \ref{lem:3.2}.

\subsection{General case}
\label{sec:3.4}
In the general case (no smallness assumption), the following result is useful in the proof of Lemma \ref{lem:3.2}. 
\begin{proposition}
\label{prop:3.3}
Let $s\in(3/2,2)$.
For any $M>0$ there exist $T(M)>0$ and $C_3(M)>0$ such that for $\varphi_\eps\in H^\infty$ satisfying $\norm[\varphi_\eps]_{H^s}\leq M$, smooth solutions of \eqref{eq:3.1} with $u_{\varepsilon}(0)=\varphi_\varepsilon$
satisfy
\begin{align}
\label{eq:3.13}
\sup_{\eps\in(0,1)}\sup_{t\in[0,T(M)]}\norm[u_\eps(t)]_{H^s}\leq C_3(M).
\end{align}
Moreover, there exists $C_4(M)>0$ such that 
\begin{align}
\label{eq:3.14}
\norm[u_\eps(t)]_{H^2}\leq \norm[u_\eps(0)]_{H^2}\exp(C_4(M)t),\quad t\in[0, T(M)]
\end{align}
where the constant $C_4(M)$ is independent of $\eps\in(0,1)$.
\end{proposition}
\begin{proof}
The derivation of \eqref{eq:3.13} can be done in the same way as \cite[Section 4]{TF80}. However, as a technical issue, we need to pay attention to fractional derivatives for nonlinearities with fractional powers. For the convenience of the reader, we give a self-contained proof of \eqref{eq:3.13} in Appendix \ref{sec:A}.

Once we get \eqref{eq:3.13}, it follows from the energy inequality \eqref{eq:3.6} and Sobolev's embedding that 
\begin{align*}
\frac{d}{dt} \norm[u_\eps(t)]_{H^2}^2
\cleq C_3(M)^{2\sigma}\norm[u_\eps(t)]_{H^2}^2.
\end{align*}
Applying Gronwall's lemma, we conclude \eqref{eq:3.14}.
\end{proof}
\begin{proof}[Proof of Lemma \ref{lem:3.2}]
Fix some $s\in(3/2,2)$. We set
\begin{align*}
M_*=\sup_{t\in[0,T_{\rm max})}\norm[u(t)]_{H^1}+1
\end{align*}
and for any fixed  $T\in(0,T_{\rm max})$ we set
\begin{align*}
M=\sup_{t\in[0,T]}\norm[u(t)]_{H^s}+1.
\end{align*}
Let $u_\eps$ be the smooth solution of \eqref{eq:3.1} with $u_\eps(0)=J_\eps\varphi$. By Proposition \ref{prop:3.3} there exists $T_0=T_0(M)>0$ such that 
\begin{align*}
\sup_{\eps\in(0,1)}\sup_{t\in[0,T_0]}\norm[u_{\eps}(t)]_{H^s}\leq C_3(M)
\end{align*} 
and 
\begin{align}
\label{eq:3.15}
\norm[u_\eps(t)]_{H^2}\leq \norm[u_\eps(0)]_{H^2}\exp(C_4(M)t)
\leq \norm[\varphi]_{H^2}\exp(C_4(M)t),\quad t\in[0,T_0].
\end{align}
Since we obtained $H^2$ boundedness for $u_\eps$, it follows from Lemma \ref{lem:3.1} that
\begin{align*}
\norm[u-u_\eps]_{C([0,T_0], H^s)} \to 0\quad\text{as}~\eps\downarrow0.
\end{align*}
In particular, there exists $\eps_1\in(0,1)$ such that for any $\eps\in(0,\eps_1)$
\begin{align*}
\norm[u_\eps(T_0)]_{H^s}\leq\norm[u(T_0)]_{H^s}+1\leq M.
\end{align*}
Next, we apply Proposition \ref{prop:3.3} with $u_\eps(T_0)$ as the initial data of \eqref{eq:3.1}. Thus, we obtain
\begin{align*}
\norm[u_\eps(t+T_0)]_{H^2}\leq \norm[u_\eps(T_0)]_{H^2}\exp(C_4(M)t)
\leq \norm[\varphi]_{H^2}\exp\bigl(C_4(M)(t+T_0)\bigr),
\quad t\in[0,T_0], 
\end{align*}
where we have used \eqref{eq:3.15} in the last inequality. Therefore, we obtain
\begin{align*}
\norm[u_\eps(t)]_{H^2}\leq \norm[\varphi]_{H^2}\exp(C_4(M)t),\quad t\in[0,2T_0],
\end{align*}
and deduce by Lemma \ref{lem:3.1} that
\begin{align*}
\norm[u-u_\eps]_{C([0,2T_0], H^s)} \to 0\quad\text{as}~\eps\downarrow0.
\end{align*}
In particular, there exists $\eps_2\in(0,\eps_1)$ such that for any $\eps\in(0,\eps_2)$
\begin{align*}
\norm[u_\eps(2T_0)]_{H^s}\leq\norm[u(2T_0)]_{H^s}+1\leq M.
\end{align*}
Iterating this argument for a finite number of times, namely at most $[\frac {T}{T_0}]+1$, we deduce that there exists $\eps_*\in (0,1)$ such that for any $\eps\in(0,\eps_*)$
\begin{align*}
\norm[u_\eps(t)]_{H^2}\leq \norm[\varphi]_{H^2}\exp(C_4(M)t),\quad t\in[0,T].
\end{align*}
From Lemma \ref{lem:3.1} again we obtain
\begin{align*}
\norm[u-u_\eps]_{C([0,T], H^1)} \to 0\quad\text{as}~\eps\downarrow0.
\end{align*}
By choosing $\eps_*$ possibly smaller, we deduce that
\begin{align*}
\sup_{\eps\in(0,\eps_*)}\sup_{t\in[0,T]}\norm[u_\eps(t)]_{H^1}\leq M_*.
\end{align*}
This completes the proof.
\end{proof}

\subsection{Comments on our proof}
\label{sec:3.5}

In order to prove Theorem \ref{thm:1.1} via the $H^2$ identity in Theorem \ref{thm:2.4}, we need to derive the uniform boundedness of approximate solutions in $H^1$ from the assumption \eqref{eq:3.8}. When the initial data is small, the $H^1$ norm of approximate solutions can be uniformly controlled as discussed in Section \ref{sec:3.3}, but this cannot be expected in general for the large data. To prove the uniform boundedness of $u_\eps$ in $H^1$ for the general case, it would be necessary to show that $u$ and $u_\eps$ are reasonably close in the $H^1$ topology. It is sufficient to be able to prove the convergence
\begin{align*}
\norm[u-u_\eps]_{C([0,T],H^1)}\underset{\eps\downarrow0}{\longrightarrow}0\quad\text{for any}~T\in(0,T_{\rm max}),
\end{align*}
but it is not easy to see whether this can be proved just from the information about the $H^1$ boundedness of $u$. This is closely related to the fact that the wellposedness of $H^1$ has not been proved yet for \eqref{eq:1.1}.

Our strategy is to split the time interval and obtain the required boundedness through the local Cauchy theory in $H^2$. However, if we try to apply the $H^2$ local theory directly, we need to show that $u$ and $u_\eps$ are close in the $H^2$ topology on the extension argument, which would require quite a lot of calculations. To avoid this complicated issue, we improve the local Cauchy theory in \cite{AS15} and more specifically prove the uniform boundedness of approximate solutions in $H^s$ for $s\in(3/2,2)$.
%
Proposition \ref{prop:3.3} guarantees that the time width on the extension argument can be taken depending on the $H^s$ norm, which implies that we only need to prove the difference estimate between $u$ and $u_\eps$ in the $L^2$ topology from a viewpoint of interpolation. In order to show uniform estimates in $H^s$, it is necessary to calculate fractional derivatives for fractional nonlinearities, but this calculation would be more economical than the difference estimate in $H^2$.


\section{Heuristic arguments on modified energies}
\label{sec:4}

In this section we will explain how modified energies in Theorem \ref{thm:2.4} were derived from a heuristic discussion. We shall use the notation $A_1(u)\sim A_2(u)$, with 
$A_1(u), A_2(u)$ functionals depending on $u$, to denote the fact that $A_1(u)-A_2(u)$
is a good term in the sense of Definition \ref{def:2.3}.
For the solution of \eqref{eq:1.1}, by a formal calculation (the detailed computation for approximate solutions $u_\varepsilon$ is done along Lemma
\ref{lem:2.6}) we first obtain the relation
\begin{align*}
\frac{d}{dt}\norm[\pt^2 u]_{L^2}^2\sim {-}4B_1(u)-2B_2(u).
\end{align*} 
The bad terms $B_1(u)$ and $B_2(u)$ (see \eqref{eq:2.6} and \eqref{eq:2.7}, respectively) are obstacles when one derives a priori estimates on $H^2$ by using Gronwall's lemma. The key is to find suitable correction terms that can eliminate these bad terms so that
\begin{align}
\label{eq:4.1}
\frac{d}{dt}\l[ \norm[\pt^2 u]_{L^2}^2+\text{(correction terms)}\r]\sim 0.
\end{align}
We note that by the equation \eqref{eq:1.1} we are allowed to replace
$\pt^2u$  by $-i \pt_t u$ (indeed the replacement of $\pt^2u$  by the nonlinear contribution coming for the equation involves less derivatives and provides always good terms). Under this observation, we can do the following formal manipulations, up to harmless multiplicative constants:
\begin{align}
\label{eq:4.2}
\begin{aligned}
B_{1}(u)&\rightarrow\Re\int\pt^2u\pt^2\widebar{u}\pt u\widebar{u}|u|^{2(\sigma-1)}
\\
&\rightarrow-\frac 12 \Re\int i\pt_tu\pt^2\widebar{u}\pt u\widebar{u}|u|^{2(\sigma-1)}
+ \frac 12 \Re\int i \pt^2 u\pt_t \widebar{u}\pt u\widebar{u}|u|^{2(\sigma-1)}
\\&
\rightarrow \frac 12 \frac{d}{dt}\Im\int \pt^2\widebar{u}\pt u|u|^{2\sigma}
-\frac 12 \frac{d}{dt}\Im\int \pt^2u\pt u\widebar{u}^{2}|u|^{2(\sigma-1)}+\hbox{(other terms)},
\\[3pt]
B_2(u)&\rightarrow \Re\int \pt^2u \pt^2u {\pt\widebar{u}}\widebar{u}|u|^{2(\sigma-1)}
\\
&\rightarrow \Re\int -i \pt_t u \pt^2u {\pt\widebar{u}}\widebar{u}|u|^{2(\sigma-1)}
\rightarrow\frac{d}{dt}\Im\int \pt^2 u \pt \bar u |u|^{2\sigma}+\hbox{(other terms)}.
\end{aligned}
\end{align}
Thus one can see that the two terms 
\begin{align*}
\Im\int \pt^2\widebar{u}\pt u|u|^{2\sigma},
\quad
\Im\int \pt^2u\pt u\widebar{u}^2|u|^{2(\sigma-1)}
\end{align*}
appear as the possible correction terms in order to cancel out $B_1(u)$ and $B_2(u)$.
Following this heuristic, we go backward and compute the time derivative of the candidate correctors above following the rule that $\pt_t u$ will be replaced by $i\pt^2 u$.
In view of Lemma \ref{lem:2.5} we can neglect along our computation all the integrands which are, up to conjugate, 
either the product of 
$\pt^2 u$, $(\pt u)^3$ and other factors without derivatives, or 
the product  of $(\pt^2 u)^2$ and other factors without derivatives.
Noting these things and using integration by parts, we obtain
\begin{align*}
\frac{d}{dt}\Im\int\pt^2\widebar{u}\pt u|u|^{2\sigma}
&\sim -2B_1(u)-2B_2(u),
\\
\frac{d}{dt}\sigma\Im\int \pt^2u\pt u\widebar{u}^2|u|^{2(\sigma-1)}
&\sim (\sigma+1)B_2(u) +3B_3(u),
\intertext{where the third bad term $B_3(u)$ is defined by \eqref{eq:2.13}. Fortunately, $B_3(u)$ is handled with another correction term as}
\frac{d}{dt}\sigma(\sigma-1)\Im\int(\pt u)^3\widebar{u}^{3}|u|^{2(\sigma-2)}
&\sim -6B_3(u).
\end{align*}
The third correction term can be found by a similar heuristic argument as in \eqref{eq:4.2}. 
Collecting the above calculations, for $\alpha,\beta\in\R$ we obtain
\begin{align*}
&\frac{d}{dt}
\begin{aligned}[t]
&\Bigl( \norm[\pt^2u]_{L^2}^2-\alpha\Im\int\pt^2\widebar{u}\pt u|u|^{2\sigma} -\beta\sigma\Im\int \pt^2u\pt u\widebar{u}^2|u|^{2(\sigma-1)}
\\
&{}\qquad+\frac{\beta}{2}\sigma(\sigma-1)\Im\int(\pt u)^3\widebar{u}^{3}|u|^{2(\sigma-2)}\Bigr)
\end{aligned}
\\ 
\sim&\,(2\alpha-4)B_1(u)+\l(2\alpha-\beta(\sigma+1)-2\r)B_2(u).
\end{align*}
If we set
\begin{align*}
\alpha=2,~\beta =\frac{2}{\sigma+1},
\end{align*}
then all the coefficients of the last bad terms are canceled out. Hence \eqref{eq:4.1} holds.


\appendix
\section{Local uniform bounds in $H^s$}
\label{sec:A}

In this section we give a complete proof of \eqref{eq:3.13} in Proposition \ref{prop:3.3}. We prove it by following the flow \cite[Section 4]{TF80}.

We define the fractional derivative by
\begin{align*}
(D^su)(x)= \sum_{n\in\Z}|n|^s\hat{u}(n)e^{2\pi i nx},\quad x\in\T.
\end{align*}
We use the following classical result on fractional derivatives.
\begin{lemma}[{\cite[Lemma 1.1]{ST76}}]
\label{lem:A.1}
Let $s>1$ and $\gamma>1/2$. For $u,v\in H^s(\T)$, we have
\begin{align}
\label{eq:A.1}
\norm[D^s(uv)-uD^sv]_{L^2}\cleq \norm[u]_{H^s}\norm[v]_{H^\gamma}+\norm[u]_{H^{\gamma+1}}\norm[v]_{H^{s-1}}.
\end{align}
\end{lemma}
Let $s\in(3/2,2)$. For the approximate solution $u_\eps$ to \eqref{eq:3.1},
\begin{align*}
\frac{d}{dt}\norm[D^su_\eps]_{L^2}^2&=2\Im\rbra[i\pt_t D^su_\eps,D^su_\eps]
=-2\Im\rbra[D^s g_\eps(u_\eps), D^su_\eps],
\end{align*}
where we have used the notation \eqref{eq:2.2}. 
We use the notation $v_\eps=J_\eps u_\eps$ and rewrite the last term as
\begin{align}
\label{eq:A.2}
-2\Im\rbra[D^s g(v_\eps)-i|v_\eps|^{2\sigma}D^s\pt v_\eps, D^sv_\eps]-2\Im\rbra[i|v_\eps|^{2\sigma}D^s\pt v_\eps, D^sv_\eps].
\end{align}
Applying \eqref{eq:A.1} with $\gamma=s-1$ to the first term, we obtain
\begin{align}
\label{eq:A.3}
\abs[ {\rbra[D^s g(v_\eps)-i|v_\eps|^{2\sigma}D^s\pt v_\eps, D^sv_\eps]} ]\cleq\norm[|v_\eps|^{2\sigma}]_{H^s}\norm[v_\eps]_{H^s}
\norm[D^s v_\eps]_{L^2}.
\end{align}
Now we need to calculate $\norm[D^s(|v_\eps|^{2\sigma})]_{L^2}$. For this purpose, the following characterization of the homogeneous Sobolev norm is convenient.
\begin{lemma}[{\cite[Proposition 1.3]{BO13}}]
\label{lem:A.2}
Let $\gamma\in(0,1)$. Then, for $u\in \dot{H}^\gamma(\T^d)$ we have the relation 
  \begin{align*}
  \norm[u]_{\dot{H}^\gamma(\T^d)}^2 \sim 
  \iint_{\T^d\times[-\frac{1}{2},\frac{1}{2})^d} \frac{|u(x+y)-u(x)|^2}{|y|^{d+2\gamma}}dxdy,
  \end{align*}
Here the notation $A\sim B$ means that both $A\cleq B$ and $B\cleq A$ hold true.
\end{lemma}
Inspired from the argument of \cite[Section 4]{CHO}, we prove the following result.
\begin{lemma}
\label{lem:A.3}
Let $s\in(3/2,2)$. For $u\in H^s(\T)$, we have
\begin{align}
\label{eq:A.4}
\norm[D^s (|u|^{2\sigma})]_{L^2}
\cleq\norm[u]_{L^\infty}^{2(\sigma-1)}\norm[\pt u]_{L^\infty}\norm[D^{s-1}u]_{L^2}+\norm[u]_{L^\infty}^{2\sigma-1}\norm[D^su]_{L^2}.
\end{align}
\end{lemma}
\begin{proof}
We set $\gamma=s-1$. We use the decomposition
\begin{align*}
D^s=\cH\pt D^{s-1} =\cH D^{\gamma}\pt,
\end{align*}
where $\cH$ is the Hilbert transform. Note that
\begin{align*}
\pt(|u|^{2\sigma})=\sigma|u|^{2(\sigma-1)}\l(\pt u\widebar{u}+u\pt\widebar{u} \r).
\end{align*}
Thus, we have
\begin{align*}
\norm[D^s(|u|^{2\sigma})]_{L^2}
=\norm[D^\gamma\pt(|u|^{2\sigma})]_{L^2}
\leq2\sigma\norm[D^\gamma(|u|^{2\sigma-1}\widebar{u}\pt u)]_{L^2}.
\end{align*}
We set $f(u)=|u|^{2(\sigma-1)}\widebar{u}$. Applying Lemma \ref{lem:A.2}, we obtain
\begin{align*}
\norm[D^\gamma\l(f(u)\pt u\r) ]_{L^2}^2
&\sim
\iint_{\T\times[-\frac{1}{2},\frac{1}{2})} 
|y|^{-1-2\gamma}\abs[(f(u)\pt u)(x+y)-(f(u)\pt u)(x) ]^2
dxdy
\\
&\cleq
\begin{aligned}[t]
&\iint_{\T\times[-\frac{1}{2},\frac{1}{2})} 
|y|^{-1-2\gamma}\abs[\l(f(u)(x+y)-f(u)(x)\r) \pt u(x+y) ]^2
dxdy
\\
&{}+\iint_{\T\times[-\frac{1}{2},\frac{1}{2})} 
|y|^{-1-2\gamma}\abs[ \l(\pt u(x+y)-\pt u(x)\r) f(u)(x) ]^2
dxdy.
\end{aligned}
\end{align*}
Note that
\begin{align*}
\abs[ f(u)-f(v)] \cleq \l( |u|^{2(\sigma-1)}+|v|^{2(\sigma-1)}\r)|u-v|
\quad\text{for}~u,v\in\C.
\end{align*}
By H\"older's inequality and Lemma \ref{lem:A.2},
the RHS of the above inequality is estimated by
\begin{align*}
&\cleq
\begin{aligned}[t]
&\norm[u]_{L^\infty}^{4(\sigma-1)}\norm[\pt u]_{L^\infty}^2\iint_{\T\times[-\frac{1}{2},\frac{1}{2})} 
|y|^{-1-2\gamma}\abs[ u(x+y)-u(x)]^2
dxdy
\\
&{}+\norm[f(u)]_{L^\infty}^2\iint_{\T\times[-\frac{1}{2},\frac{1}{2})} 
|y|^{-1-2\gamma}\abs[\pt u(x+y)-\pt u(x)]^2
dxdy
\end{aligned}
\\
&\cleq \norm[u]_{L^\infty}^{4(\sigma-1)}\norm[\pt u]_{L^\infty}^2\norm[D^\gamma u]_{L^2}^2+\norm[u]_{L^\infty}^{2(2\sigma-1)}\norm[D^{\gamma+1}u]_{L^2}^2.
\end{align*}
This completes the proof.
\end{proof}
By \eqref{eq:A.3}, Lemma \ref{lem:A.3}, and Sobolev's embedding, the first term in \eqref{eq:A.2} is estimated as
\begin{align*}
\abs[ {2\Im\rbra[D^s g(v_\eps)-i|v_\eps|^{2\sigma}D^s\pt v_\eps, D^sv_\eps]} ]\cleq \norm[v_\eps]_{H^1}^{2\sigma-1}\norm[ v_\eps]_{H^s}^3.
\end{align*}
Regarding the second term in \eqref{eq:A.2}, by integration by parts and Sobolev's embedding, 
\begin{align*}
-2\Im\rbra[i|v_\eps|^{2\sigma}D^s\pt v_\eps, D^sv_\eps]
&=\rbra[\pt(|v_\eps|^{2\sigma}), |D^sv_\eps|^2]
\\
&\cleq\norm[v_\eps]_{L^\infty}^{2\sigma-1}\norm[\pt v_\eps]_{L^\infty} \norm[D^s v_\eps]_{L^2}^2
\cleq  \norm[v_\eps]_{H^1}^{2\sigma-1}\norm[ v_\eps]_{H^s}^3.
\end{align*}
Gathering these estimates, we obtain
\begin{align*}
\frac{d}{dt}\norm[D^su_\eps]_{L^2}^2\cleq \norm[v_\eps]_{H^1}^{2\sigma-1}\norm[ v_\eps]_{H^s}^3\cleq\norm[ u_\eps]_{H^s}^{2\sigma+2}.
\end{align*}
Combined with conservation of the $L^2$ norm, this yields 
\begin{align*}
\frac{d}{dt}\norm[u_\eps]_{H^s}^2\cleq \norm[ u_\eps]_{H^s}^{2\sigma+2}.
\end{align*}
Therefore, by a simple differential inequality, there exists $c>0$ independent of $\eps\in(0,1)$ such that
\begin{align*}
\norm[u_\eps(t)]_{H^s}^2\leq
\l( \frac{1}{ \norm[u_\eps(0)]_{H^s}^{-2\sigma} -ct }\r)^{1/\sigma}
\leq\l( \frac{1}{M^{-2\sigma}-ct } \r)^{1/\sigma}.
\end{align*}
Hence, \eqref{eq:3.13} follows by choosing $T(M)=M^{-2\sigma}/(2c)$.

\section*{Acknowledgments}

N.V. and M.H. are supported by the Italian MIUR PRIN project 2020XB3EFL and M.H. is also supported by JSPS KAKENHI Grant Number JP22K20337.
T.O. is supported by JSPS KAKENHI Grant Numbers 18KK073 and 24H00024.
%



\end{document}